\documentclass[a4paper,11pt]{amsart}
\usepackage{amsmath}
\usepackage{amsthm}
\usepackage{amsfonts}
\usepackage{amssymb}
\usepackage{latexsym}
\usepackage{mathrsfs}

\usepackage{soul}

\usepackage[usenames]{color}
\usepackage[all]{xy}

\usepackage{graphicx}
\usepackage{latexsym}
\usepackage[cp850]{inputenc}
\usepackage{epsfig}
\usepackage{psfrag}

\newtheorem{theorem}{Theorem}[section]{\bf}{\it}
\newtheorem{lemma}[theorem]{Lemma}{\bf}{\it}
\newtheorem{proposition}[theorem]{Proposition}{\bf}{\it}
\newtheorem{corollary}[theorem]{Corollary}{\bf}{\it}
{\bf}{\it} 
{\bf}{\it}
\newtheorem*{theorem*}{Theorem}

\newtheorem{remark}[theorem]{Remark}
{\bf}{\it}
{\bf}{\it}

\newtheorem*{namedtheorem}{\theoremname}
\newcommand{\theoremname}{testing}
\newenvironment{named}[1]{\renewcommand{\theoremname}{#1}\begin{namedtheorem}}{\end{namedtheorem}}

\newtheorem*{definition*}{Definition}

{\bf}{\it}
{\bf}{\it}
{\bf}{\it}
\newtheorem{example}[theorem]{Example}
\newtheorem*{example*}{Example}

\theoremstyle{remark}

\theoremstyle{definition}

\theoremstyle{remark}

\newcommand{\R}{\mathbb R}
\newcommand{\Z}{\mathbb Z}
\newcommand{\C}{\mathbb C}
\newcommand{\N}{\mathbb N}

\newcommand{\loc}{{\operatorname{loc}}}

\newdimen\vintkern\vintkern11pt
\def\vint{-\kern-\vintkern\int}

\newcommand{\norm}[1]{\lVert #1 \rVert}

\newcommand{\grad}{\nabla}

\newcommand{\bH}{\mathbb{H}}

\newcommand{\vol}{\mathrm{vol}}

\newcommand{\Ratio}{\mathsf{R}}

\newcommand{\QR}{\mathsf{QR}}

\title[]{Quasiregular curves}

\author{Pekka Pankka}
\address{Department of Mathematics and Statistics, P.O. Box 68 (Pietari Kalmin katu 5), FI-00014 University of Helsinki, Finland}
\email{pekka.pankka@helsinki.fi}

\date{\today}

\thanks{This work was supported in part by the Academy of Finland project \#297258.}
\date{\today}

\begin{document}

\begin{abstract}
We extend the notion of a pseudoholomorphic vector of Iwaniec, Verchota, and Vogel to mappings between Riemannian manifolds. Since this class of mappings contains both quasiregular mappings and (pseudo)holomorphic curves, we call them quasiregular curves.
 
Let $n\le m$ and let $M$ be an oriented Riemannian $n$-manifold, $N$ a Riemannian $m$-manifold, and $\omega \in \Omega^n(N)$ a smooth closed non-vanishing $n$-form on $N$. 
A continuous Sobolev map $f\colon M \to N$ in $W^{1,n}_\loc(M,N)$ is a \emph{$K$-quasiregular $\omega$-curve for $K\ge 1$} if $f$ satisfies the distortion inequality $(\norm{\omega}\circ f)\norm{Df}^n \le K (\star f^* \omega)$ almost everywhere in $M$.

We prove that quasiregular curves satisfy Gromov's quasiminimality condition and a version of Liouville's theorem stating that bounded quasiregular curves $\R^n \to \R^m$ are constant. We also prove a limit theorem that a locally uniform limit $f\colon M \to N$ of $K$-quasiregular $\omega$-curves $(f_j \colon M\to N)$ is also a $K$-quasiregular $\omega$-curve.

We also show that a non-constant quasiregular $\omega$-curve $f\colon M \to N$ is discrete and satisfies $\star f^*\omega >0$ almost everywhere, if one of the following additional conditions hold: the form $\omega$ is simple or the map $f$ is $C^1$-smooth. 
\end{abstract}

\maketitle


\section{Introduction}
Quasiconformal homeomorphisms admit three classical definitions: \emph{analytic definition}, based on weak differential, \emph{geometric definition}, based on modulus of curve families, and \emph{metric definition} based on infinitesimal metric distortion. Out of these three ways to define quasiconformality, the metric definition is the only one which does not require the spaces to have the same dimension and, in particular, allows us to consider quasiconformal embeddings into higher dimensional spaces. The geometric definition, which is based on comparison of moduli of curve families and their images, is ineffective in this case, since curve families contained in a lower dimensional subspace typically have zero modulus. The analytic definition, which extends to the definition of quasiregular mappings, is based on the Jacobian determinant of the mapping and hence is \emph{a priori} not at our disposal. 

The higher dimensional quasiconformal theory has an extensive literature. We refer to e.g.~monographs of V\"ais\"al\"a \cite{Vaisala-book} or Gehring, Martin, and Palka \cite{Gehring-Martin-Palka-book} or articles of Heinonen and Koskela \cite{Heinonen-Koskela-Invent, Heinonen-Koskela-Acta} and V\"ais\"al\"a \cite{Vaisala-TAMS-1981} for discussion on quasiconformal and related quasisymmetric theory.

In this article, we discuss an extension of the analytic definition for quasiregular mappings, called quasiregular curves, similar to pseudoholomorphic vectors of Iwaniec, Verchota, and Vogel \cite{Iwaniec-Verchota-Vogel}. The name stems from the observation that holomorphic and pseudoholomorphic curves are quasiregular curves. 
 
Recall that a continuous mapping $f\colon M\to N$ between oriented Riemannian $n$-manifolds is \emph{$K$-quasiregular} for $K\ge 1$ if $f$ belongs to the Sobolev space $W^{1,n}_\loc(M,N)$ and satisfies the distortion inequality 
\[
\norm{Df}^n \le K J_f 
\]
almost everywhere in $M$, where $\norm{Df}$ is the operator norm of the differential $Df$ of $f$ and $J_f$ the Jacobian determinant of $f$ defined by $f^*\vol_N = J_f \vol_M$. For homeomorphisms this is the analytic definition of quasiconformality and therefore a quasiregular homeomorphism is called  \emph{quasiconformal}. We refer to monographs of Reshetnyak \cite{Reshetnyak-book}, Rickman \cite{Rickman-book}, and Iwaniec--Martin \cite{Iwaniec-Martin-book} for the theory of quasiregular mappings.

For the definition of a quasiregular curve, we define first the auxiliary notion of an $n$-volume form on an $m$-manifold for $m \ge n$. Let $M$ and $N$ be an oriented Riemannian $n$-manifold and an Riemannian $m$-manifold, respectively, for $n \le m$. We say that a smooth differential $n$-form $\omega \in \Omega^n(N)$ is an \emph{$n$-volume form} if $\omega$ is non-vanishing and closed. Note that, since $\omega \wedge \star \omega$ is a non-vanishing $m$-form, the manifold $N$ is orientable. Here, and in what follows, $\Omega^n(N)$ is the space of smooth differential $n$-forms on a smooth manifold $N$.

In the following definition, the spaces $M$ and $N$ are an oriented Riemannian $n$-manifold and a Riemannian $m$-manifold, respectively, for $n\le m$, and $\omega \in \Omega^n(N)$ is an $n$-volume form.

\begin{definition*}
A continuous map $f\colon M \to N$ is a \emph{$K$-quasiregular $\omega$-curve for $K\ge 1$} if $f$ belongs to the Sobolev space $W^{1,n}_\loc(M, N)$ and 
\begin{equation}
\label{eq:QRC}
\tag{QRC}
(\norm{\omega}\circ f) \norm{Df}^n \le K \left( \star f^*\omega \right)
\end{equation}
almost everywhere on $M$.
\end{definition*}
Here $\star f^*\omega$ is the Hodge star dual of the $n$-form $f^*\omega$, that is, the function satisfying $(\star f^*\omega)\vol_M = f^*\omega$. The function $\norm{\omega}\colon N \to [0,\infty)$ is the \emph{pointwise comass norm of $\omega$} given by
\[
\norm{\omega}(p) = \max\{ |\omega_p(v_1,\ldots, v_k)| \colon v_1,\ldots,v_k \in T_p N,\ |v_i| \le 1\}
\]
for each $p\in N$; see Federer \cite[Section 1.8.1]{Federer-book}.

\begin{remark}
In \cite{Iwaniec-Verchota-Vogel}, Iwaniec, Verchota, and Vogel define that a map $f=(f_1,\ldots, f_n) \colon \Omega \to \C^n$ , is a \emph{pseudoholomorphic vector on a domain $\Omega \subset \C$} if $f$ belongs to the Sobolev space $W^{1,2}_\loc(\Omega, \C^n)$ and satisfies the distortion inequality $|Df|^2 \le 2K \left( J_{f_1}+\cdots J_{f_n}\right)$ almost everywhere for $K\ge 1$, where $|Df|$ is the Hilbert--Schmidt norm of $Df$. Since $J_{f_1}+\cdots + J_{f_n} = f^*\omega$ for the standard symplectic form $\omega = dx_1\wedge dy_1 + \cdots + dx_n \wedge dy_n$ and norms $\norm{Df}$ and $|Df|$ are equivalent, we have that pseudoholomorphic vectors are quasiregular curves. We refer to \cite[Section 7]{Iwaniec-Verchota-Vogel} for more details.
\end{remark}

Extending the introduced terminology, we also say that $f\colon M\to N$ is a \emph{quasiregular $\omega$-curve} if $f$ is a $K$-quasiregular $\omega$-curve for some $K\ge 1$, and that $f\colon M\to N$ is a \emph{quasiregular curve} if $f$ is a quasiregular $\omega$-curve for some $n$-volume form $\omega\in \Omega^n(N)$ on $N$. In these cases, we tacitly assume without further notice that the manifold $M$ is an oriented Riemannian $n$-manifold and $N$ is a Riemannian $m$-manifold for $n \le m$.

\begin{example}
For oriented Riemannian manifolds $M$ and $N$ of same dimension and for $\omega = \vol_N$, we recover the definition of a $K$-quasiregular map $M\to N$. Thus quasiregular maps are quasiregular curves.  In the same vein, if $\pi \colon P \to N$ is a Riemannian bundle over $N$ and $F\colon M \to N$ is a $K$-quasiregular $\omega$-curve for $\omega =\pi^*\vol_N$, then the composition $f=\pi \circ F \colon M \to N$ is a $K$-quasiregular mapping. Indeed, since $\pi$ is a Riemannian isometry, the map $f$ is in $W^{1,n}_\loc(M,N)$ and we have the estimate
\begin{align*}
\norm{Df}^n 
&= (\norm{\pi^*\vol_N}\circ \pi \circ F)\norm{D(\pi \circ F)}^n 
\le (\norm{\omega} \circ F) \norm{DF}^n \\
&\le K (\star F^*\omega) 
= K (\star F^*\pi^*\vol_N) = K (\star f^*\vol_N) 
= K J_f. 
\end{align*}
\end{example}

\begin{example}
For $j=1,2$, let $N_j$ be a Riemannian $n$-manifold, let $\omega_j \in \Omega^n(N_j)$ be an $n$-volume form, $f_j \colon M \to N_j$ a $K$-quasiregular map, and $\pi_j \colon N_1\times N_2 \to N_j$ a projection. Let $\omega = \pi_1^*\omega_1 + \pi_2^*\omega_2 \in \Omega^n(N_1\times N_2)$. Then $f=(f_1,f_2) \colon M\to N_1\times N_2$ is a $K$-quasiregular $\omega$-curve. Indeed, since $\norm{Df} \le \norm{Df_1} + \norm{Df_2}$ and $\star f^*\omega = \star f_1^*\omega_1 + \star f_2^*\omega_2$ almost everywhere in $M$, and $\norm{\omega}=1$, we have that 
\begin{align*}
(\norm{\omega}\circ f) \norm{Df}^n 
&\le 2^n \left( \norm{Df_1}^n + \norm{Df_2}^n \right)
\le 2^n K (\star f^*\omega).
\end{align*}
By the same argument, holomorphic curves $f=(f_1,\ldots, f_n) \colon \Omega\to \C^n$, where $\Omega \subset \C$ is a domain, are $1$-quasiregular curves. Indeed, since $\norm{Df}^2 \le \norm{Df_1}^2 + \cdots + \norm{Df_n}^2$, we have that $f$ is a $1$-quasiregular $\omega$-curve for the symplectic form $\omega = dx_1\wedge dy_1 + \cdots + dx_n \wedge dy_n$. 
\end{example}

\begin{example}
Let $(N,\omega, J)$ be a K\"ahler manifold and suppose that the almost complex structure $J$ is calibrated by the symplectic form $\omega$. Suppose further that $\omega$ is bounded and  $\ell(\omega) = \inf_{z\in N} \ell(\omega)_p > 0$, where $\ell(\omega)_p = \min_{|v|=1} \omega(v,iv)$ for each $p\in N$. Then a $J$-holomorphic curve $f \colon \C \to N$ is a $K$-quasiregular $\omega$-curve for $K=\norm{\omega}_\infty/\ell(\omega)$. Indeed, since $J$ is an isometry and $J \circ Df = Df \circ i$, we have, for each $z\in \C$ and each unit vector $v\in T_z\C$, that $\norm{Df}^2 = |Df(v)|^2$. Thus, for an orthonormal basis $\{e_1,e_2\}$ of $T_z\C$ at $z\in \C$, we have that
\begin{align*}
\star f^*\omega &= f^*\omega(e_1,e_2) = f^*\omega(e_1,ie_1) = \omega(Df(e_1), Df (ie_1)) \\
&= \omega(Df(e_1), J Df(e_1)) \ge \ell(\omega) |Df(e_1)|^2.
\end{align*}
For more discussion, we refer to Gromov's article \cite{Gromov-Invent} on pseudoholomorphic curves in symplectic geometry or e.g.~Audin and Lafontaine \cite{Audin-Lafontaine-book} for details.
\end{example}

\begin{remark}
Examples of $n$-volume forms on $m$-manifolds for $n \le m$ are e.g.~exterior powers of symplectic forms and coclosed contact forms. More precisely, if $N$ has even dimension $2n$ and $\omega\in \Omega^2(N)$ is a symplectic $2$-form, then $\omega^{\wedge k}$ is a $2k$-volume form on $N$. In this case, $\omega^{\wedge n}$ is a standard volume form on $N$ and quasiregular $\omega^{\wedge n}$-curves into $N$ are quasiregular mappings. 

If $N$ has odd dimension $2n+1$ and $\theta \in \Omega^1(N)$ is a contact form satisfying $d(\star \theta) = 0$, then $\omega = \star \theta$ is an $2n$-volume form. For example, the Heisenberg form $\theta_H = dt - \frac{1}{2}(x dy - ydx)$ in $\R^3$ is a coclosed contact form. Clearly, there exists an abundance of quasiregular $(\star \theta_H)$-curves $B^2 \to \R^3$. However, we do not know if there exist non-constant entire quasiregular $(\star \theta_H)$-curves $\R^2 \to \R^3$. Note that here the $2$-form $\star \theta_H$ is simple.
\end{remark}

We note in passing that, similarly as quasiconformal or quasiregular maps, the distortion of quasiregular curves is conformally invariant in the following sense: \emph{Let $f\colon M \to N$ be a $K$-quasiregular curve between Riemannian manifolds $(M,g_M)$ and $(N,g_N)$. Then $f$ is $K$-quasiregular with respect to Riemannian manifolds $(M,\tilde g_M)$ and $(N,\tilde g_N)$ for Riemannian metrics $\tilde g_m$ and $\tilde g_N$ conformally equivalent to $g_M$ and $g_N$, respectively}. Therefore, for example, the space 
\[
\QR_K(M,N;\omega) = \{ f\colon M \to N \colon f \text{ is a } K\text{-quasiregular } \omega\text{-curve}\}
\]
of all $K$-quasiregular $\omega$-curves between Riemannian manifolds $M$ and $N$ for a fixed $n$-volume form $\omega\in \Omega^n(N)$, is a conformal invariant of manifolds $M$ and $N$.

\bigskip

In this article, we prove three results on quasiregular curves for general $n$-volume forms and one in the special case of simple $n$-volume forms.

\subsection{Quasiminimality of quasiregular curves}

The first of the three theorems we prove on general quasiregular curves is that a quasiregular $\omega$-curve is quasiminimal in the sense of Gromov's definition \cite[Definition 6.37]{Gromov-book} if the form $\omega$ has bounded ratio
\[
\Ratio(\omega) = \frac{\sup \norm{\omega}}{\inf\norm{\omega}} < \infty.
\]

For the definition of quasiminimality, we give first an auxiliary definition of a competitor. Let $f\colon M \to N$ be a continuous map in $W^{1,n}_\loc(M,N)$ and let $W \Subset M$ be a compact $n$-submanifold with boundary. We say that a continuous map $h \colon M \to N$ is an \emph{competitor for $f$ on $W$} (or \emph{$(f,W)$-competitor} for short) if $h$ is a Sobolev map in $W^{1,n}_\loc(M,N)$, $f|_{\partial W} = h|_{\partial W}$, and $fW$ is homologous to $hW$ in $N$ modulo $f(\partial W)$. 

\begin{definition*}
A continuous $W^{1,n}_\loc(M,N)$-mapping $f\colon M \to N$ from an $n$-manifold $M$ to an $m$-manifold $N$ for $m \ge n$ is \emph{$C$-quasiminimal} if, for each compact $n$-submanifold $W \Subset M$ with boundary, each $(f,W)$-competitor $h \colon M \to N$ satisfies
\[
\int_W \norm{\wedge^n Df} \vol_M \le C \int_W \norm{\wedge^n Dh} \vol_M.
\]
\end{definition*}

Quasiregular $\omega$-curves are quasiminimal, quantitatively, if $\omega$ has bounded ratio. More precisely, we have the following result.

\begin{theorem}
\label{thm:quasiminimality-intro}
Let $\omega \in \Omega^n(N)$ be an $n$-volume form of bounded ratio. Then a $K$-quasiregular $\omega$-curve $M \to N$ is $K\Ratio(\omega)$-quasiminimal.
\end{theorem}

\subsection{Liouville's theorem for quasiregular curves}

Liouville's classical theorem in complex analysis states that \emph{bounded entire functions $\C\to \C$ are constant}. It was known from very early on that the same result holds also for quasiregular mappings $\R^n \to \R^n$; see e.g.~\cite[Corollary III.1.14]{Rickman-book} and the related discussion.  A version of Liouville's theorem holds also for quasiregular curves. 

\begin{theorem}
\label{thm:Liouville-intro}
Let $N$ be a complete Riemannian $m$-manifold and $\omega \in \Omega^n(N)$ an exact $n$-volume form for $n\le m$. Then each bounded quasiregular $\omega$-curve $\R^n \to N$ is constant.
\end{theorem}

As for quasiregular mappings, the proof reduces to a simple application of the $n$-parabolicity of the Euclidean $n$-space and a Caccioppoli inequality (Proposition \ref{prop:Caccioppoli}) for quasiregular curves. 

\begin{remark}
Another version of Liouville's theorem 
states that \emph{a quasiregular $\omega$-curve $f\colon M\to N$ is constant if $M$ is a closed manifold and $\omega$ is an exact form}. Indeed, since $f^*\omega$ is a weakly exact $n$-form on $M$, we have
\[
\int_M (\norm{\omega}\circ f) \norm{Df}^n \le K \int_M f^*\omega = 0.
\]
Since $\norm{\omega}\circ f$ is a non-negative function, we obtain that $Df= 0$ almost everywhere and that $f$ is constant. In particular, quasiregular curves from closed manifolds into Euclidean spaces are constant.
\end{remark}

\subsection{Limit theorem}

Our second theorem is a limit theorem for quasiregular curves. For quasiregular mappings the statement reads as follows \cite[Theorem VI.8.6]{Rickman-book}: \emph{a locally uniform limit of $K$-quasiregular mappings is $K$-quasiregular}. For quasiregular curves, an analogous statement holds.

\begin{theorem}
\label{thm:limit-intro}
For $n\le m$, let $M$ and $N$ be an oriented Riemannian $n$-manifold and a Riemannian $m$-manifold, respectively, let $\omega \in \Omega^n(N)$ be an $n$-volume form on $N$, and let $(f_j)$ be a sequence of $K$-quasiregular $\omega$-curves $f_j \colon M \to N$ converging locally uniformly to a mapping $f\colon M \to N$. Then $f$ is a $K$-quasiregular $\omega$-curve.
\end{theorem}

A short comment on the proof is in order. We may mostly follow the (classical) proof for quasiregular mappings in \cite{Rickman-book}. However, since we do not have local index theory at our disposal, we obtain the sharp distortion constant for the limit map by modifying the argument in \cite[Theorem 8.7,1]{Iwaniec-Martin-book}.


\subsection{Quasiregular curves for simple volume forms and Reshetnyak's theorem}

An $n$-form $\omega \in \Omega^n(N)$ is \emph{simple} (or \emph{decomposable}) if there exist $1$-forms $\omega_1,\ldots, \omega_n \in \Omega^1(N)$ for which $\omega = \omega_1\wedge \cdots \wedge \omega_n$. 

Quasiregular curves for simple volume forms have particularly simple structure: \emph{locally they are graphs over quasiregular mappings}. For simplicity, we state this result for quasiregular curves between in Euclidean spaces. 

\begin{theorem}
\label{thm:graph-intro}
Let $f\colon \Omega\to \R^m$ be a $K$-quasiregular $\omega$-curve, where $\Omega$ is a domain in $\R^n$, $n\le m$, $\varepsilon>0$, and $K'>K$. Then, for each $x\in \Omega$, there exists a neighborhood $D\Subset M$ of $x$, an isometry $L \colon \R^m \to \R^m$, a $K'$-quasiregular map $\hat f \colon D\to \R^n$, and a continuous Sobolev map $h \colon D\to \R^{m-n}$ in $W^{1,n}(D,\R^{m-n})$ for which $F = L \circ f|_D = (\hat f, h) \colon D \to \R^n \times \R^{m-n}$ and  
\[
(\star f^*\omega)/((1+\varepsilon)K') \le \norm{\omega_{f(x)}} J_{\hat f} \le (1+\varepsilon)K (\star f^*\omega)
\]
almost everywhere in $D$.
\end{theorem}

Having this local description at our disposal, we obtain a version of Reshetnyak's theorem in the case of a simple $n$-volume form. Recall that Reshetnyak's theorem for quasiregular mappings states that \emph{a non-constant quasiregular mappings is discrete and open}. A mapping $f\colon M\to N$ is \emph{discrete} if, for each $y\in N$, the fiber $f^{-1}(y)$ is a discrete set in $M$, and \emph{open} if the image $fU$ of an open set $U\subset M$ is open in $N$.

\begin{remark}
Before discussing the positive result, we emphasize that Reshetnyak's theorem fails for quasiregular curves in general. Indeed, in \cite{Iwaniec-Verchota-Vogel} Iwaniec, Verchota, and Vogel construct a Lipschitz regular pseudoholomorpic vector $F=(f_1,f_2)\colon \C \to \C^2$, which is constant on the lower half-plane but satisfies $J_{f_1} + J_{f_2} \equiv 1$ almost everywhere on the upper half-plane; see \cite[Lemma 5]{Iwaniec-Verchota-Vogel}. As a quasiregular curve, the map $F$ constructed in \cite{Iwaniec-Verchota-Vogel} has distortion $K> 2$. Iwaniec, Verchota, and Vogel show that such pseudoholomorphic vectors $\Omega \to \C^n$, where $\Omega \subset \C$ is a domain, do not exist if the distortion $K$ -- in the sense of quasiregular curves -- is close to $1$. We refer to \cite[p.150]{Iwaniec-Verchota-Vogel} for a detailed discussion.
\end{remark}

Regarding the openness in Reshetnyak's theorem, we note that it is immediate from the definition that, due to increase of dimension, quasiregular curves are not open mappings. Simple examples also show that quasiregular curves are not even interior mappings. Recall that a mapping $f\colon M\to N$ is \emph{interior} if the image $f\Omega$ of an open set $\Omega\subset M$ is open in the induced topology of the image $fM \subset N$. 

\begin{example*}
Let $p\in \Z_+$ and let $h \colon \C \to \R$ be a smooth function satisfying $|h(z)| \le |z|^p$ and $|h'(z)| \le p |z|^{p-1}$ for all $z\in \C$. Then the map $f\colon \C \to \R^3$, $z\mapsto (z^p,h(z))$, where $\R^3 = \C\times \R$, is a quasiregular $\omega$-curve for $\omega = dx\wedge dy$. However, for a generic choice of $h$, the curve $f$ is not interior.
\end{example*}

After these disclaimers, we are now ready to state a positive result. For the statement, we say that a map $f\colon M \to N$ is \emph{locally quasi-interior at $x\in M$} if $x$ has a neighborhood $D\Subset M$ for which $f(x)$ is in the interior of $fU$, with respect to $fD$, for each neighborhood $U \subset D$ of $x$. 

\begin{corollary}
\label{cor:discreteness-intro}
Let $f\colon M\to N$ be a non-constant quasiregular $\omega$-curve, where $\omega$ is a simple $n$-volume form. Then $f$ is discrete and locally quasi-interior at each point. 
\end{corollary}

As a consequence of Theorem \ref{thm:graph-intro}, we also obtain that quasiregular curves for simple $n$-volume forms have analytic properties similar to quasiregular mappings.

\begin{corollary}
\label{cor:analytic-intro}
Let $f\colon M \to N$ be a non-constant quasiregular $\omega$ for a simple $n$-volume form $\omega$ in $N$. Then
\begin{enumerate}
\item (positivity of the Jacobian) $\star f^*\omega >0$ almost everywhere in $M$, 
\item (higher integrability) there exists $p=p(n,K)>0$ for which $f\in W^{1,p}_\loc(M,N)$, and
\item (differentiability) $f$ is differentiable almost everywhere.
\end{enumerate}
\end{corollary}

\begin{remark}
Since $n$-volume forms of codimension $1$ are simple, we have that these results hold in particular for all codimension $1$ quasiregular curves $M \to N$, that is, when $\dim N = 1+ \dim M$. In particular, quasiregular curves $\R^2 \to \R^3$ have the properties in Corollaries \ref{cor:discreteness-intro} and \ref{cor:analytic-intro}. This is contrast to mappings associated to more general null Lagrangians; see Iwaniec, Verchota, and Vogel \cite[Lemma 6]{Iwaniec-Verchota-Vogel}.
\end{remark}

\subsection*{$C^1$-smooth quasiregular curves}

We end this introduction with a discussion on Reshetnyak's theorem for $C^1$-smooth quasiregular curves. It is an elementary observation that a $C^1$-smooth quasiregular curve $f \colon M \to N$ is locally a quasiregular curve with respect to a simple $n$-volume form. Indeed, since the question is local it suffices to consider a $K$-quasiregular curve $f\colon \Omega \to \R^m$ defined on a domain $\Omega \subset \R^n$. Let $x\in \Omega$. Then, by continuity of $Df$ and $\omega$, we may fix a neighborhood $U$ of $x$ and a multi-index $J=(j_1,\ldots, j_n)$ for which we have the estimate 
\[
\star f^*\omega = \sum_{I} (u_I\circ f) (\star f^*(dx_I)) \le 2{m \choose n} (u_J\circ f) (\star f^*(dx_J))
\]
in $U$, where we denote $dx_I = dx_{i_1}\wedge \cdots \wedge dx_{i_n}$ for each multi-index $I=(i_1,\ldots, i_n)$. Since $\norm{u_J dx_J} \le \norm{\omega}$, we conclude that $f|_U \colon U \to N$ is a $2{m \choose n} K$-quasiregular $(u_J dx_J)$-curve.  

Theorem \ref{thm:graph-intro} now yields that, locally, $C^1$-smooth quasiregular curves are graphs over quasiregular maps and, in particular, discrete maps by Corollary \ref{cor:discreteness-intro}. We summarize this observation as follows.
\begin{corollary}
\label{cor:smooth-Reshetnyak}
A non-constant $C^1$-smooth quasiregular $\omega$-curve $f\colon M \to N$ is a discrete map satisfying $\star f^*\omega>0$ almost everywhere in $M$.
\end{corollary}

\bigskip

This article is organized as follows. In Sections \ref{sec:quasiminimality}, \ref{sec:Liouville}, and \ref{sec:limit}, we prove Theorems \ref{thm:quasiminimality-intro}, \ref{thm:Liouville-intro}, and \ref{thm:limit-intro}, respectively. Finally, in Section \ref{sec:simple-forms}, we prove Theorem \ref{thm:graph-intro} and its corollaries.

\subsection*{Acknowledgements}
We thank Daniel Faraco for an important comment at the right time and pointing us to  article \cite{Iwaniec-Verchota-Vogel}. We also thank Kari Astala, Mario Bonk, David Drasin, Jani Onninen, Jang-Mei Wu, and Xiao Zhong for comments and discussions on these topics.


\section{Quasiregular curves are quasiminimal}
\label{sec:quasiminimality}

In this section we show that quasiregular curves satisfy Gromov's (homological) quasiminimality criterion \cite[Definition 6.36]{Gromov-book} if the $n$-volume form has bounded ratio. 

\begin{named}{Theorem \ref{thm:quasiminimality-intro}}
Let $\omega \in \Omega^n(N)$ be an $n$-volume form of bounded ratio. Then a $K$-quasiregular $\omega$-curve $f\colon M \to N$ is $K\Ratio(\omega)$-quasiminimal.
\end{named}

\begin{proof}
Let $W \Subset M$ be an $n$-manifold with boundary and let $h \colon M \to N$ be an $(f,W)$-competitor. Since $fW$ and $hW$ are homologous modulo $f(\partial W)$, there exists an $(n+1)$-chain $\Sigma$ for which $\partial \Sigma = fW - hW$, as chains. By de Rham's theorem, we may identify the duality pairing of the $n$-form $\omega$ with the $n$-chains $hW$ and $fW$ as integration. Thus we have that 
\[
\int_W f^*\omega - \int_W h^*\omega = \int_{fW} \omega - \int_{hW}\omega = \int_{\partial \Sigma} \omega = \int_\Sigma d\omega = 0.
\]

Since $\norm{\wedge^n Df} \le \norm{Df}^n$ and $\star h^*\omega \le (\norm{\omega}\circ h) \norm{\wedge^n Dh}$ almost everywhere, we have that
\begin{align*}
\int_W \norm{\wedge^n Df} \vol_M &\le 
\int_W \norm{Df}^n \vol_M 
\le \int_W \frac{(\norm{\omega}\circ f)}{\min_N \norm{\omega}} \norm{Df}^n \vol_M \\
&\le \frac{K}{\min_N \norm{\omega}} \int_W f^*\omega 
= \frac{K}{\min_N \norm{\omega}}\int_W h^*\omega \\
&\le \frac{K}{\min_N \norm{\omega}}\int_W (\norm{\omega}\circ h) \norm{\wedge^n Dh} \vol_M \\
&\le K \Ratio(\omega) \int_W \norm{\wedge^n Dh} \vol_M.
\end{align*}
We conclude that
\[
\int_W \norm{\wedge^n Df}^n \vol_M \le K \Ratio(\omega) \int_W \norm{\wedge^n Dh} \vol_M.
\]
\end{proof}

\begin{remark}
The proof of Theorem \ref{thm:quasiminimality-intro} is essentially the same as Gromov's argument in \cite[Example 6.3.7]{Gromov-book} for quasiminimality of the graph $Gf \colon M \to M\times N$, $x\mapsto (x,f(x))$, of a quasiregular mapping $f\colon M \to N$. The form $\omega$ in Gromov's argument is $\omega = \pi_M^*\vol_M + \pi_N^* \vol_N$, where $\pi_M \colon M\times N \to M$ and $\pi_N \colon M\times N \to N$ are the natural projections.
\end{remark}


\section{Liouville's theorem for entire quasiregular curves}
\label{sec:Liouville}

In this section, we prove a version of the Liouville's theorem. 

\begin{named}{Theorem \ref{thm:Liouville-intro}}
Let $N$ be a complete Riemannian $m$-manifold and $\omega \in \Omega^n(N)$ an exact $n$-volume form for $n\le m$. Then each bounded quasiregular $\omega$-curve $\R^n \to N$ is constant.
\end{named}

As for quasiregular mappings, the proof of Liouville's theorem is an application of Caccioppoli's inequality, which we formulate here as follows.

\begin{proposition}
\label{prop:Caccioppoli}
Let $f\colon M \to N$ be a $K$-quasiregular $\omega$-curve for an exact $n$-volume form $\omega\in \Omega^n(N)$ and let $\tau\in \Omega^{n-1}(N)$ be a potential of $\omega$, that is, $\omega = d\tau$. Then there exists a constant $C=C(n) >0$ having the property that, for every non-negative function $\psi \in C^\infty_0(M)$, 
\[
\int_M \psi^n f^*\omega \le C K^{n-1} \int_M |\grad \psi|^n \left( \frac{\norm{\tau}^n}{\norm{\omega}^{n-1}}\right) \circ f.
\]
\end{proposition}

\begin{proof}
Let $\zeta = \psi^n$. Then, by Stokes' theorem, 
\[
\int_M \zeta f^*\omega = \int_M \zeta d f^*\tau = \int_M d(\zeta f^*\tau) - \int_M d \zeta \wedge f^*\tau = - \int_M d\zeta \wedge f^*\tau. 
\]
Hence, by pointwise norm estimates,  
\begin{align*}
\int_M \zeta f^*\omega 
&\le C\int_M |\grad \zeta| (\norm{\tau}\circ f) |Df|^{n-1} 
\le C n \int_M \psi^{n-1}|\grad \psi| (\norm{\tau}\circ f) |Df|^{n-1},
\end{align*}
where $C=C(n)>0$. By H\"older's inequality, 
\begin{align*}
\int_M \zeta f^*\omega 
&\le C n \left( \int_M |\grad \psi|^n \frac{(\norm{\tau}\circ f)^n}{(\norm{\omega}\circ f)^{n-1}} \right)^{\frac{1}{n}} 
\left( \int_M \psi^n (\norm{\omega} \circ f) |Df|^n \right)^{\frac{n-1}{n}} \\
&\le C n K^{\frac{n-1}{n}} \left( \int_M |\grad \psi|^n \frac{(\norm{\tau}\circ f)^n}{(\norm{\omega}\circ f)^{n-1}} \right)^{\frac{1}{n}} 
\left( \int_M \zeta f^*\omega \right)^{\frac{n-1}{n}}.
\end{align*}
Thus
\[
\int_M \zeta f^*\omega \le C n^n K^{n-1} \int_M |\grad \psi|^n \frac{(\norm{\tau}\circ f)^n}{(\norm{\omega}\circ f)^{n-1}}.
\]
\end{proof}

Liouville's theorem is now an almost immediate consequence.

\begin{proof}[Proof of Theorem \ref{thm:Liouville-intro}]
Suppose that $f$ is bounded. It suffices to show that, for every $r>0$, we have 
\[
\int_{B^n(r)} f^*\omega = 0.
\]
Then $\norm{Df}= 0$ almost everywhere and $f$ is constant in $B^n(r)$ by the Poincar\'e inequality. 

Let $r>0$ and $\varepsilon>0$. Since $\mathrm{cap}_n(\bar B^n(r),\R^n) =0$, there exists $\psi \in C^\infty_0(\R^n)$ for which $\psi|_{B^n(r)} \equiv 1$ and 
\[
\int_{\R^n} |\grad \psi|^n \le \varepsilon.
\]
Since $\omega$ is exact, we may fix a potential $\tau\in \Omega^{n-1}(\R^m)$ of $\omega$. Since $N$ is complete and $f$ is bounded, we have that $f\R^n \Subset N$. Since $\tau$ is smooth and $\omega$ is smooth and non-vanishing, we further have that the function $\norm{\tau}^n/\norm{\omega}^{n-1}$ is bounded on $f\R^n$. Thus, by Caccioppoli's inequality, there exists a constant $C>0$ for which
\[
\int_{B^n(r)} f^*\omega \le \int_{\R^n} \psi^n f^*\omega \le C \int_\Omega |\grad \psi|^n \le C\varepsilon.
\]
The claim follows.
\end{proof}

\begin{remark}
\label{rmk:Liouville}
The previous Liouville's theorem admits a following variation: \emph{Let $N$ be a Riemannian $m$-manifold and $\omega\in \Omega^{n}(N)$ an $n$-volume form with a potential $\tau\in \Omega^{n-1}(N)$ for which the function $\norm{\tau}^n/\norm{\omega}^{n-1}$ is bounded. Then each quasiregular $\omega$-curve $\R^n \to N$ is constant.} 
\end{remark}

\begin{remark}
The version of Liouville's theorem in Remark \ref{rmk:Liouville} shows that for each non-zero $(n-1)$-covector $\zeta \in \wedge^{n-1}\R^{m-1}$ and $n$-volume form $\omega_0 = x_m^{-n} \zeta \wedge dx_m\in \Omega^n(\bH^m)$, a quasiregular $\omega$-curve $\R^n \to \bH^m$ is constant. For simplicity, suppose that $\zeta = dx_1 \wedge \cdots \wedge dx_{n-1}$. Then, in the upper half-space model $\bH^m = \R^{m-1}\times (0,\infty)$ of the hyperbolic $m$-space, we have that the $(n-1)$-form $\tau_0 = (-1)^n (n-1)^{-1} x_m^{1-n} dx_1\wedge \cdots \wedge dx_{n-1}$ is one of the potentials of $\omega_0$. Since $\norm{dx_1\wedge \cdots \wedge dx_{n-1}} = x_m^{n-1}$ and $\norm{dx_1\wedge \cdots \wedge dx_{n-1}\wedge dx_m} = x_m^n$, we have that $\norm{\tau_0} = (m+1)^{-1}$ and $\norm{\omega_0} = 1$. In particular, $\norm{\tau_0}^n/\norm{\omega_0}^{n-1}$ is bounded. The case for general $n$-covector $\zeta$ is similar. 

Note that there are easy examples of $n$-volume forms on $\bH^n$, which admit non-constant quasiregular curves from $\R^n$. For example, let $f\colon \R^n \to \R^n$ be a $K$-quasiregular map and fix $t>0$. Then the map $F=(f,0,t) \colon \R^n \to \R^n \times \R^{m-n-1}\times (0,\infty)$ is a $K$-quasiregular $\omega$-curve for $\omega = dx_1\wedge \cdots \wedge dx_n$. Clearly, the map $F$ is not a quasiregular $\omega_0$-curve. In fact, $F^*\omega_0 = 0$ almost everywhere.
\end{remark}


\section{Limit theorem}
\label{sec:limit}

In this section, we prove Theorem \ref{thm:limit-intro} which states that a locally uniform limit of $K$-quasiregular $\omega$-curves is also a $K$-quasiregular $\omega$-curve. Since the result is local, it suffices to prove the following local result.

\begin{theorem}
\label{thm:limit}
Let $\Omega \subset \R^n$ be a domain and let $(f_j)$ be a sequence of $K$-quasiregular $\omega$-curves $f_j \colon \Omega \to \R^m$ converging locally uniformly to a mapping $f\colon \Omega \to \R^m$. Then $f$ is a $K$-quasiregular $\omega$-curve.
\end{theorem}

\begin{proof}[Proof of Theorem \ref{thm:limit-intro} assuming Theorem \ref{thm:limit}]
To show that the limiting map has the same distortion as the maps in the sequence, let $a\in \N$ be an auxiliary parameter.
Let now $\{(\Omega_\alpha,\varphi_\alpha)\}_\alpha$ and $\{ (V_\beta,\psi_\beta)\}_\beta$ be atlases of $M$ and $N$, respectively, consisting of $(1+1/a)$-bilipschitz charts and having the property that, for each index $\alpha$, there exists an index $\beta$ for which $f\Omega_\alpha \Subset V_\beta$. Existence of such atlases follow from the exponential maps $TM \to M$ and $TN \to N$ of $M$ and $N$, respectively, and continuity of $f$.

By Theorem \ref{thm:limit} and the chain rule in each $\Omega_\alpha$, we obtain that $f$ is in $W^{1,n}_\loc(\Omega_\alpha, N)$ for each $\alpha$ and that 
\[
(\norm{\omega}\circ f) \norm{Df}^n \le K (1+1/a)^{4n} f^*\omega
\]
almost everywhere in $\Omega_\alpha$ for each $\alpha$, and hence almost everywhere in $M$. 

Thus, almost everywhere in $M$, we have that
\[
(\norm{\omega}\circ f) \norm{Df}^n \le K f^*\omega
\]
as claimed.
\end{proof}

The proof of Theorem \ref{thm:limit} follows the idea of the same result for quasiregular maps; see Rickman's book \cite[Section VI.8]{Rickman-book}.

We separate the first part of the proof as a separate lemma and show that locally uniform limits of quasiregular curves are in the right Sobolev class. As in the case of quasiregular maps, this is essentially an application of the Caccioppoli inequality (Proposition \ref{prop:Caccioppoli}).
\begin{lemma}
\label{lemma:limit_is_Sobolev}
Let $f \colon \Omega \to \R^m$ be a locally uniform limit of a sequence $(f_j)$ of $K$-quasiregular $\omega$-curves $f_j \colon \Omega \to \R^m$. Then $f\in W^{1,n}_\loc(\Omega, \R^m)$ and, for each domain $U \Subset \Omega$, there exists a subsequence $(f_{i_j})$ of $(f_j)$ converging weakly to $f$ in $W^{1,n}(U,\R^m)$.
\end{lemma}

\begin{proof}
Let $U \Subset \Omega$ be a domain and $\psi \in C^\infty_0(\Omega)$ a non-negative function satisfying $\psi|_U \equiv 1$. Let $W \Subset \Omega$ be a domain containing the support of $\psi$. Since $(f_j)$ converges locally uniformly, there exists a domain $V \Subset \R^m$ containing all images $f_jW$ and $fW$. 

Since $\omega$ is closed, it is exact. Let $\tau\in \Omega^{n-1}(\R^m)$ be a potential of $\omega$, that is, $d\tau = \omega$. Since $V$ has compact closure, we have that $y \mapsto \norm{\tau(y)}^n/\norm{\omega(y)}^{n-1}$ is a bounded function on $V$. Thus, we have by the Caccioppoli estimate (Proposition \ref{prop:Caccioppoli}) that there exists a constant $C=C(n,\omega|_V, K, \psi)>0$ for which 
\begin{align*}
\min_{y\in V} \norm{\omega_y} \int_U \psi^n \norm{Df_j}^n &\le 
\int_U \psi^n (\norm{\omega}\circ f_j) \norm{Df_j}^n \le \int_\Omega \psi^n K f_j^*\omega \le C.
\end{align*}
for all $j\in \N$. Since $\min_{y\in V}\norm{\omega_y} > 0$, we have that $(f_j)$ is a bounded sequence in $W^{1,n}(U,\R^m)$. By weak compactness, there exists a subsequence $(f_{j_i})$ converging weakly in $W^{1,n}(U,\R^m)$ to a map $\hat f \colon U \to \R^m$. Since $f_j \to f$ in $L^n(U,\R^m)$, we have in addition that $f=\hat f$. Thus $f\in W^{1,n}(U,\R^m)$. We refer to \cite[Proposition VI.7.9]{Rickman-book} for details.
\end{proof}

\begin{lemma}
\label{lemma:omega_convergence}
Let $f \colon \Omega \to \R^m$ be a locally uniform limit of a sequence $(f_j)$ of $K$-quasiregular $\omega$-curves $f_j \colon \Omega \to \R^m$. Then $f^*_j \omega \to f^*_j \omega$ weakly, that is, for each non-negative $\zeta \in C^\infty_0(\Omega)$,  
\begin{equation}
\label{eq:omega_convergence}
\int_\Omega \zeta f_j^*\omega \to \int_\Omega \zeta f^*\omega
\end{equation}
as $j\to \infty$.
\end{lemma}

\begin{proof}
Let $\zeta\in C^\infty_0(\Omega)$ be non-negative function and let $U \Subset \Omega$ be a domain containing the support of $\zeta$. Since $f_j \to f$ locally uniformly, we may also fix a domain $V\Subset \R^m$ which contains the union $fU \cup \bigcup_j f_jU$.

Since $\omega$ is closed, it is exact, that is, $\omega = \sum_J d(\tau_J dx_J)$, where $J=(j_1,\ldots, j_{n-1})$ is a $(n-1)$-multi-index and, for each $J$, $\tau_J\in C^\infty(\R^m)$. For each $J$, let also $\omega_J = d\tau_J \wedge dx_J$. Then $\omega = \sum_J \omega_J$ and it suffices to prove \eqref{eq:omega_convergence} for each $\omega_J$.

Let $J$ be an $(n-1)$-multi-index and set $u_1 = \tau_J$ and $u_i = x_{i_{i-1}}$ for each $i\in \{1,\ldots, n-1\}$. Then $\omega_J = du_1 \wedge \cdots \wedge du_n$. For each $i=1,\ldots, n$, we denote $h_i = u_i \circ f$ and further, for each $j\in \N$, we set $h_{i,j} = u_i \circ f_j$. Then $f^*\omega = dh_1 \wedge \cdots \wedge dh_n$ and $f_j^*\omega = dh_{i,j} \wedge \cdots \wedge dh_{n,j}$. 

For the standard telescoping argument based on integration by parts, we observe first that
\begin{align*}
f_j^*\omega_J - f^*\omega_J 
&= dh_{1,j}\wedge \cdots \wedge dh_{n,j} - dh_1 \wedge \cdots \wedge dh_n \\
&= \sum_{k=1}^n dh_{1,j}\wedge dh_{k-1,j} \wedge (dh_{k,j}-dh_k) \wedge dh_{k+1,j} \wedge \cdots \wedge dh_n \\
&= \sum_{k=1}^n dh_{1,j}\wedge dh_{k-1,j} \wedge d(h_{k,j}-h_k) \wedge dh_{k+1,j} \wedge \cdots \wedge dh_n.
\end{align*}

Since the form $dh_{1,j}\wedge \cdots \wedge dh_{k-1,j} \wedge d(\zeta(h_{k,j}-h_k)) \wedge dh_{k+1,j} \wedge \cdots \wedge dh_n$ is exact and compactly supported in $\Omega$, we have the telescoping equality
\begin{align*}
&\int_\Omega \zeta \left( f_j^*\omega_J- f^*\omega_J\right) \\
&\quad = \sum_{k=1}^n \int_\Omega dh_{1,j}\wedge dh_{k-1,j} \wedge \zeta d(h_{k,j}-h_k) \wedge dh_{k+1,j} \wedge \cdots \wedge dh_n \\ 
&\quad = \sum_{k=1}^n \int_\Omega (h_k - h_{k,j}) dh_{1,j}\wedge dh_{k-1,j} \wedge d\zeta \wedge dh_{k+1,j} \wedge \cdots \wedge dh_n.
\end{align*}

As usual, we have now a pointwise inequality 
\begin{align*}
&|dh_{1,j}\wedge dh_{k-1,j} \wedge d\zeta \wedge dh_{k+1,j} \wedge \cdots \wedge dh_n| \\
&\quad \le |f_j^*du_1| \cdots |f_j^*du_{k-1}|\cdot |d\zeta|\cdot |f^*du_{k+1}| \cdots |f^*du_n| \\
&\quad \le (\max_k |\grad u_k|_{L^\infty(V)})^{n-1} |\grad \zeta| \norm{Df_j}^{k-1} \norm{Df}^{n-k} 
\end{align*}
almost everywhere in $\Omega$. Thus, by H\"older's inequality, we have the estimate
\begin{align*}
&\left| \int_\Omega \zeta \left( f_j^*\omega_J- f^*\omega_J\right) \right| \\
&\le C\left( \int_U \norm{Df_j}^{k-1} \norm{Df}^{n-k} \right)\norm{h-h_j}_{L^\infty(U)} \\
&\le C \left( \int_U \norm{Df_j}^{n} \right)^{(k-1)/n} \left( \int_U \norm{Df}^{n\frac{n-k}{n-k+1}} \right)^{(n-k+1)/n}\norm{h-h_j}_{L^\infty(U)} \\
&\le C \left( \int_U \norm{Df_j}^{n} \right)^{(k-1)/n} \left( \int_U \norm{Df}^{n} \right)^{(n-k)/n}\norm{h-h_j}_{L^\infty(U)},
\end{align*}
where constant $C=C(u_1,\ldots, u_n,\zeta,U)$ depends only on norms of $u_1,\ldots, u_m$ and $\grad \zeta$, and on the volume of $U$. By Caccioppoli's inequality the sequence $(f_j|_U)$ is bounded in $W^{1,n}(U,\R^n)$. Since $\norm{h-h_j}_{L^\infty(U)} \to 0$ as $j \to \infty$, the claim follows.
\end{proof}

We are now ready to finish the proof of the limit theorem (Theorem \ref{thm:limit}). So far we have followed the strategy in \cite[Section VI.8]{Rickman-book}. To obtain the sharp constant, we move now to follow the proof with the argument of Iwaniec and Martin \cite[Theorem 8.7.1]{Iwaniec-Martin-book} for the same theorem.  We do not know if the method in the proof of \cite[Theorem VI.8.6]{Rickman-book} admits an adaptation in our current setting.
 
We separate the proof for the lower semicontinuity of the operator norm from the argument of Iwaniec and Martin as a separate lemma.

\begin{lemma}
\label{lemma:lower_sc}
Let $\Omega \subset \R^n$ be a domain and let $(f_j)$ be a sequence in $W^{1,n}_\loc(\Omega, \R^m)$, which converges weakly to a map $f\in W^{1,n}_\loc(\Omega,\R^m)$. Then, for each domain $U \Subset \Omega$,
\[
\int_U \norm{Df}^n \le \liminf_{j \to \infty} \int_U \norm{Df_j}^n.
\]
\end{lemma}
\begin{proof}
Let $\varphi \in C^\infty_0(\Omega)$ a non-negative function satisfying $\varphi|_U \equiv  1$.

Following Iwaniec and Martin, we fix measurable unit vector fields $\xi \colon \Omega \to \R^n$ and $\zeta \colon \Omega \to \R^m$ satisfying 
\[
\norm{Df(x)} = \max_{|v|=1} |Df(x)v| = |Df(x) \xi(x)| = \langle Df(x) \xi(x), \zeta(x) \rangle
\]
almost everywhere. Then, by the convexity of the function $t\mapsto t^n$, we have that
\begin{align*}
\norm{Df_j}^n - \norm{Df}^n &\ge n \norm{Df}^{n-1}\left( \norm{Df_j}-\norm{Df} \right) \\
&\ge n \norm{Df}^{n-1} \langle Df_j \xi - Df \xi, \zeta \rangle \\
&= n \langle Df_j-Df, \norm{Df}^{n-1} \xi \otimes \zeta \rangle
\end{align*}
where $\xi\otimes \zeta \colon \Omega \to \R^{m\times n}$ is the matrix field $x\mapsto \zeta(x) \xi(x)^\top$.

Since $\norm{Df}^{n-1}\in L^{n/(n-1)}_\loc(\Omega)$ and $\xi$ and $\zeta$ have pointwise unit length, we have that $\norm{Df}^{n-1} \xi\otimes \zeta \in L^{n/(n-1)}_\loc(\Omega, \R^{m\times n})$. Since $Df_j \to Df$ weakly in $L^n(U,\R^{m\times n})$ and $\norm{Df}^{n-1} \xi\otimes \zeta \in L^{n/(n-1)}(U,\R^{m\times n})$, we have that
\[
\int_U \langle Df_j-Df, \norm{Df}^{n-1} \xi \otimes \zeta \rangle \to 0
\]
as $j \to \infty$. Thus
\[
\int_U \norm{Df}^n \le \liminf_{j \to \infty} \int_U \norm{Df_j}^n.
\]
\end{proof}

\begin{proof}[Proof of Theorem \ref{thm:limit}]
By  Lemma \ref{lemma:limit_is_Sobolev}, we have that $f\in W^{1,n}_\loc(\Omega, \R^m)$. Thus it suffices to show that the distortion inequality
\[
(\norm{\omega}\circ f)\norm{Df}^n \le K f^*\omega
\]
holds almost everywhere in $\Omega$.

Let now $x\in \Omega$ and $0<\varepsilon<\norm{\omega(x)}$. Since $\omega$ is continuous, we may fix a Euclidean ball $G=B^m(f(x),R) \Subset \R^m$ for which $\max_G \norm{\omega} - \min_G \norm{\omega} < \varepsilon$. Since $f_j \to f$ locally uniformly, we may, by passing to a subsequence, fix a Euclidean ball $B=B^n(x,r) \Subset \Omega$ for which the set $fB\cup \bigcup_j f_jB$ is compactly contained in $G$. Let now $\varphi \in C^\infty_0(B)$ be a non-negative function satisfying $\varphi|_B \equiv 1$.

By passing to a subsequence if necessary, we may assume, again by Lemma \ref{lemma:limit_is_Sobolev}, that $Df_j \to Df$ weakly in $W^{1,n}_\loc(\Omega, \R^{m\times n})$. Hence, by Lemmas \ref{lemma:lower_sc} and \ref{lemma:omega_convergence}, we have that
\begin{align*}
\int_B (\norm{\omega}\circ f) \norm{Df}^n &\le  \norm{\omega}_{L^\infty(G)} \int_B \norm{Df}^n \\
&\le \norm{\omega}_{L^\infty(G)} \liminf_{j \to \infty} \int_B \norm{Df_j}^n \\
&\le \frac{\norm{\omega}_{L^\infty(G)}}{\norm{\omega}_{L^\infty(G)}-\varepsilon} \liminf_{j \to \infty} \int_B (\norm{\omega}\circ f_j) \norm{Df_j}^n \\
&\le \frac{\norm{\omega}_{L^\infty(G)}}{\norm{\omega}_{L^\infty(G)}-\varepsilon} \liminf_{j \to \infty} \int_B K f_j^*\omega \\
&\le K \frac{\norm{\omega}_{L^\infty(G)}}{\norm{\omega}_{L^\infty(G)}-\varepsilon} \liminf_{j \to \infty} \int_\Omega \varphi f_j^*\omega\\
&= K \frac{\norm{\omega}_{L^\infty(G)}}{\norm{\omega}_{L^\infty(G)}-\varepsilon} \int_\Omega \varphi f^*\omega.
\end{align*}
Since $\varepsilon>0$ and $\varphi$ are arbitrary, we obtain the inequality
\[
\int_B (\norm{\omega}\circ f) \norm{Df}^n \le K \int_B f^*\omega.
\]
The claim now follows from Lebesgue's differentiation theorem.
\end{proof}


\section{Quasiregular curves and simple volume forms}
\label{sec:simple-forms}

In this section we consider quasiregular $\omega$-curves $M\to N$ for simple $n$-volume forms $\omega$. Recall that an $n$-form $\omega$ simple if there exists $1$-forms $\omega_1,\ldots, \omega_n$ for which $\omega = \omega_1\wedge \cdots \wedge \omega_n$. The main theorem is that such quasiregular curves are locally graphs over quasiregular maps in the following sense. 

\begin{named}{Theorem \ref{thm:graph-intro}}
Let $f\colon \Omega\to \R^m$ be a $K$-quasiregular $\omega$-curve, where $\Omega$ is a domain in $\R^n$, $n\le m$, $\varepsilon>0$, and $K'>K$. Then, for each $x\in \Omega$, there exists a neighborhood $D\Subset M$ of $x$, an isometry $L \colon \R^m \to \R^m$, a $K'$-quasiregular map $\hat f \colon D\to \R^n$, and a continuous Sobolev map $h\in W^{1,n}(D,\R^{m-n})$ for which $F = L \circ f|_D = (\hat f, h) \colon D \to \R^n \times \R^{m-n}$ and  
\[
(\star f^*\omega)/((1+\varepsilon)K') \le \norm{\omega_{f(x)}} J_{\hat f} \le (1+\varepsilon)K (\star f^*\omega)
\]
almost everywhere in $D$.
\end{named}

We begin by recalling a geometric observation. Since the proof is elementary multilinear algebra, we omit the details.

\begin{lemma}
\label{lemma:rotation}
Let $\omega \in \Omega^n(\R^m)$ be an simple $n$-volume form and $p\in N$. Then there exists an affine isometry $L \colon \R^m \to \R^m$ for which $L(p) = 0$ and $(L^{-1})^*\omega = \norm{\omega}_p dx_1\wedge \cdots \wedge dx_n$ at $0$. 
\end{lemma}

As another preparatory step, we also record a simple lemma that each quasiregular curve is locally a quasiregular curve with respect to an $n$-volume form with constant coefficients. 
\begin{lemma}
\label{lemma:constant} 
Let $f\colon \Omega \to \R^m$ be a $K$-quasiregular $\omega$-curve, $x_0\in \Omega$, and $K'>K$. Then there exists a neighborhood $\Omega' \subset \Omega$ of $x_0$ for which the restriction $f|_{\Omega'} \colon \Omega' \to \R^m$ is a $K'$-quasiregular $\omega_0$-curve, where $\omega_0$ is a constant coefficient $n$-volume form satisfying $\omega_0(f(x_0)) = \omega(f(x_0))$. 
\end{lemma}

\begin{proof}
Since $K'> K$, we may fix $c\ge 2K$ for which 
\[
\frac{c}{c-1-K} \le \frac{K'}{K}.
\]
Also, since $\omega$ is smooth and non-vanishing, we may fix a radius $\rho>0$ for which $\norm{\omega(y)-\omega_0} \le \norm{\omega_0}/c$ for all $y\in B^m(\rho)$. 

Let now $\Omega'$ be the $x_0$ component of $f^{-1}B^m(\rho)$. Then, almost everywhere in $\Omega'$, we have
\begin{align*}
\norm{\omega_0} \norm{Df}^n 
&\le \frac{c}{c-1} (\norm{\omega}\circ f) \norm{Df}^n \le  \frac{c}{c-1}K (\star f^*\omega) \\
&=  \frac{c}{c-1}K (\star f^*\omega_0) + \frac{c}{c-1}K (\star f^*(\omega-\omega_0)) \\
&\le \frac{c}{c-1}K (\star f^*\omega_0) + \frac{c}{c-1}K (\norm{\omega-\omega_0}\circ f)\norm{Df}^n \\
&\le \frac{c}{c-1}K (\star f^*\omega_0) + \frac{K}{c-1} \norm{\omega_0} \norm{Df}^n
\end{align*}
Thus
\begin{align*}
\norm{\omega_0}\norm{Df}^n 
&\le \frac{c}{c-1-K}K( f^*\omega_0) 
\le K'(f^*\omega_0)
\end{align*}
almost everywhere in $\Omega'$. The claim follows.
\end{proof}

\begin{proof}[Proof of Theorem \ref{thm:graph-intro}]
Let $x\in \Omega$. We may assume that $f(x)=0$ and that $\norm{\omega}_{f(x)} = 1$. By Lemma \ref{lemma:rotation}, there exists an isometry $L\colon \R^m \to \R^m$ for which $(L^{-1})^*\omega = dx_1\wedge \cdots \wedge dx_n$ at $0$. Then $F=(L \circ f) \colon \Omega \to \R^m$ is a $K$-quasiregular $\sigma$-curve for $\tau = (L^{-1})^*\omega$. Indeed, since $L$ is an isometry, we have that 
\[
(\norm{\sigma}\circ F) \norm{DF}^n 
= (\norm{\omega}\circ f) \norm{Df}^n \le K (\star f^*\omega) 
= K (\star F^*\sigma).
\]

By Lemma \ref{lemma:constant}, we may now fix a neighborhood $D\Subset \Omega$  of $x$ for which $F|_D \colon D\to \R^m$ is a $K'$-quasiregular $\sigma_0$-curve, where $\sigma_0$ is the constant coefficient $n$-volume form satisfying $\sigma_0(F(x)) = \sigma(F(x))$. Since $\norm{\sigma_0} = 1$, we may further assume that $1/(1+\varepsilon) \le \norm{\sigma} \le 1+\varepsilon$ on $D$.

Let $\pi \colon \R^m \to \R^n$ be the projection $(y_1,\ldots, y_m) \to (y_1,\ldots, y_n)$. Since $\hat f = \pi \circ F$, we have that
\[
\norm{D\hat f}^n = \norm{D(\pi \circ F)}^n \le \norm{DF}^n \le K' \left( \star F^*\sigma_0\right) = K' \left( \star \hat f^* \vol_{\R^n}\right)
\]
almost everywhere in $D$. Thus $\hat f$ is $K'$-quasiregular.

Since $h = \pi' \circ F \colon D \to \R^{m-n}$, where $\pi' \colon \R^m \to \R^{m-n}$ is the projection $(x_1,\ldots, x_m) \mapsto (x_{n+1},\ldots, x_m)$, we readily observe that $h\in W^{1,n}(D,\R^{m-n})$ as required.

It remains to prove that the Jacobian estimates. On one hand, we have
\begin{align*}
J_{\hat f} &= \star \hat f^*\vol_{\R^n} = \star F^*\pi^*\vol_{\R^n} = \star F^*\sigma_0\\
&\le \norm{\sigma_0} \norm{DF}^n \le (1+\varepsilon)(\norm{\omega} \circ f) \norm{Df}^n \le (1+\varepsilon)K (\star f^*\omega). 
\end{align*} 
on $D$. 
On the other hand, we have
\begin{align*}
\star f^*\omega &= \star F^*\sigma \le (\norm{\sigma}\circ F)\norm{DF}^n \le (1+\varepsilon) \norm{\sigma_0} \norm{DF}^n \\
&\le (1+\varepsilon)K' (\star F^*\sigma_0) = (1+\varepsilon)K' J_{\hat f}. 
\end{align*}
This completes the proof.
\end{proof}

The corollaries stated in the introduction are now almost immediate.

\begin{named}{Corollary \ref{cor:discreteness-intro}}
Let $f\colon M\to N$ be a non-constant quasiregular $\omega$-curve, where $\omega$ is a simple $n$-volume form. Then $f$ is discrete and locally quasi-interior at each point. 
\end{named}

\begin{proof}
Since the properties are local, it suffices to consider the case $f\colon \Omega \to \R^m$, where $\Omega \subset \R^n$ is a domain. Let $x\in \Omega$. 

By Theorem \ref{thm:graph-intro}, there exists an isometry $L\colon \R^m \to \R^m$ and a neighborhood $D\Subset \Omega$ of $x$ for which the map $F= L \circ f|_D = (\hat f,h) \colon D \to \R^n\times \R^{m-n}$ has the property that $\hat f$ is a non-constant quasiregular mapping and $h$ is a continuous Sobolev map in $W^{1,n}(D,\R^{m-n})$.

Let now $y \in \R^m$. Then $f^{-1}(y) \cap D \subset \hat f^{-1}(L(y))$. Since $\hat f$ is discrete, we conclude that $f^{-1}(y) \cap D$ is a discrete set in $D$. Thus $x$ has a neighborhood which contains only finitely many pre-images $f^{-1}(y)$ of $y$. Thus $f^{-1}(y)$ is a discrete set in $\Omega$.

To show that $f$ is locally quasi-interior, let $x\in \Omega$. Since $\hat f$ is discrete and open, we may fix a normal neighborhood $G \Subset D$ for $\hat f$ at $x$, that is, a domain satisfying $\hat f(\partial G) = \partial \hat f G$ and $\hat f^{-1}\hat f(x) = \{x\}$; see e.g.~\cite[Lemma I.4.8]{Rickman-book}. Let now $U \subset G$ be a neighborhood of $x$ contained in $G$. Then there exists another normal neighborhood $G' \subset U$ for $\hat f$ at $x$. Since components of $(\hat f^{-1}\hat fG') \cap G$ map surjectively onto $\hat fG'$ under $\hat f$, we conclude that $\hat f^{-1}\hat fG' \cap G= G'$.

Let now $\pi \colon \R^m \to \R^n$ be the projection $(x_1,\ldots, x_m) \mapsto (x_1,\ldots, x_n)$. Since $\hat fG'$ is open in $\R^n$, we conclude that $FG \cap \pi^{-1}\hat fG'$ is open in $FG$. Since $FG \cap (\pi^{-1}\hat fG') = FG'$, we obtain that $FG'$ is open in $FG$. Thus $f(x)$ is in the interior of $FU$ in $FG$. The claim follows.
\end{proof}

\begin{remark}
In a similar vein, it is a direct consequence of Theorem \ref{thm:graph-intro} that the \emph{singular set}
\[
\Sigma_f = \{ x\in M \colon f \text{ is not locally injective at } x\}
\]
of the quasiregular $\omega$-curve $f \colon M\to N$ has codimension at least $2$ if $\omega$ is simple. Indeed, since
\begin{align*}
\Sigma_{f|_D} &= \{ x\in D \colon f|_D \text{ is not locally injective at } x\} \\
&\subset  \{ x\in D \colon \hat f \text{ is not a local homeomorphism at } x\} 
= B_{\hat f},
\end{align*}
we have by the Chernavskii--V\"ais\"al\"a theorem for discrete and open maps (see \cite{Vaisala}) that $\dim \Sigma_{f|_D} \le \dim B_{\hat f} \le \dim M - 2$. 
\end{remark}

\begin{named}{Corollary \ref{cor:analytic-intro}}
Let $f\colon M \to N$ be a non-constant quasiregular $\omega$-curve for a simple volume form $\omega$ in $N$. Then
\begin{enumerate}
\item (positivity of the Jacobian) $\star f^*\omega >0$ almost everywhere in $M$, 
\item (higher integrability) there exists $p=p(n,K)>0$ for which $f\in W^{1,p}_\loc(M,N)$, and
\item (differentiability) $f$ is differentiable almost everywhere.
\end{enumerate}
\end{named}

\begin{proof}
Again, by passing to smooth $(1+\varepsilon)$-bilipschitz charts, we may assume that $f\colon \Omega \to N$ is a $K'$-quasiregular $\omega$-curve, where $\Omega \subset \R^n$ is a domain and $K'=K(1+\varepsilon)^{4n}$. Let again $L \colon \R^m \to \R^m$ be an isometry and $D\Subset \Omega$ be a domain for which $F=L\circ f = (\hat f,h) \colon D\to \R^n$, where $\hat f \colon D\to \R^m$ is a $K'$-quasiregular map and $h \colon D\to \R^m$ a continuous Sobolev map in $W^{1,n}(D,\R^{m-n})$. We may further assume that $\star F^*\sigma \le 2K' J_{\hat f}$ in $D$, where $\sigma = (L^{-1})^*\omega$.

Since $J_{\hat f}>0$ almost everywhere in $D$ by \cite[Theorem II.7.4]{Rickman-book}, we have that $\star f^*\omega >0$ almost everywhere in $D$. The first claim follows.

For the second claim, it suffices to observe that higher integrability holds for quasiregular mappings, that is, by Bojarski--Iwaniec \cite[Theorem 5.1]{Bojarski-Iwaniec}, there exists $p'=p'(n,K')>0$ for which $\hat f\in W^{1,p'}(D,\R^n)$. It remains to show that $h \in W^{1,p'}(D,\R^{m-n})$. 

Since $D \Subset \Omega$, we have that $\inf_D (\norm{\sigma}\circ F) >0$. Thus the estimate
\begin{align*}
(\norm{\sigma}\circ F) \norm{Dh}^n &\le (\norm{\sigma}\circ F) \norm{DF}^n \le K (\star F^*\sigma) \\
&\le 2KK' \norm{\omega_{f(x)}}J_{\hat f} \le 2KK' \norm{\omega_{f(x)}} \norm{D\hat f}^n 
\end{align*}
yields a bound $\norm{Dh}\le C \norm{D\hat f}$ in $D$, where $C$ depends only on $n$ and $K$. Hence $\norm{Dh}\in L^{p'}(D)$ and $f\in W^{1,p'}(D,\R^{m-n})$. The second claim follows.

Since Sobolev functions in $W^{1,p'}(D)$ for $p'>n$ are differentiable almost everywhere by the Cesari--Calder\'on lemma (see e.g.~\cite[Lemma VI.4.1]{Rickman-book}), the third claim follows. This completes the proof.
\end{proof}

We finish with Reshetnyak's theorem for $C^1$-smooth quasiregular curves.
\begin{named}{Corollary \ref{cor:smooth-Reshetnyak}}
A non-constant $C^1$-smooth quasiregular $\omega$-curve $f\colon M \to N$ is a discrete map satisfying $\star f^*\omega>0$ almost everywhere in $M$.
\end{named}

\begin{proof}
It suffices to consider the case of a $C^1$-smooth quasiregular $\omega$-curve $f\colon \Omega\to \R^m$. Then, by the discussion in the introduction, $f$ is locally a quasiregular curve with respect to a simple $n$-volume form. Thus, by Corollary \ref{cor:discreteness-intro}, $f$ is discrete. 

For the second claim, consider a domain $D \Subset \Omega$ having the property that $f|_D\colon D\to \R^m$ is a quasiregular curve with respect to a simple $n$-volume form $\omega_D$ satisfying $\norm{\omega_D} \le \norm{\omega}$ in $D$; the discussion in the introduction shows that such domains $D$ exist. Then, by (1) in Corollary \ref{cor:analytic-intro},
\[
K( \star f^*\omega) \ge (\norm{\omega}\circ f)\norm{Df}^n \ge (\norm{\omega_D}\circ f)\norm{Df}^n \ge
\star f^* \omega_D >0
\]
in $D$. The second claim follows.
\end{proof}



\end{document}